\documentclass[12pt]{article}
\usepackage{amsthm}
\usepackage{amsmath}
\usepackage{amsfonts}
\usepackage{amssymb}
\usepackage{amscd}

\usepackage[top=1in,bottom=1in,left=1.2in,right=1.2in]{geometry}

\usepackage{graphicx}
\usepackage{color}
\usepackage{url}

\newcommand{\re}{\mathbb{R}}

\newcommand{\Q}{\mathbb{Q}}

\newcommand{\diag}{\mbox{diag}}

\newcommand{\lmd}{\lambda}

\newcommand{\nn}{\nonumber}

\def\af{\alpha}

\newcommand{\reff}[1]{(\ref{#1})}
\newcommand{\pt}{\partial}

\newcommand{\mc}[1]{\mathcal{#1}}

\newcommand{\bdes}{\begin{description}}
\newcommand{\edes}{\end{description}}

\newcommand{\bal}{\begin{align}}
\newcommand{\eal}{\end{align}}

\newcommand{\bnum}{\begin{enumerate}}
\newcommand{\enum}{\end{enumerate}}

\newcommand{\bit}{\begin{itemize}}
\newcommand{\eit}{\end{itemize}}

\newcommand{\bea}{\begin{eqnarray}}
\newcommand{\eea}{\end{eqnarray}}
\newcommand{\be}{\begin{equation}}
\newcommand{\ee}{\end{equation}}

\newcommand{\baray}{\begin{array}}
\newcommand{\earay}{\end{array}}

\newcommand{\bsry}{\begin{subarray}}
\newcommand{\esry}{\end{subarray}}

\newcommand{\bca}{\begin{cases}}
\newcommand{\eca}{\end{cases}}

\newcommand{\bcen}{\begin{center}}
\newcommand{\ecen}{\end{center}}

\newcommand{\bbm}{\begin{bmatrix}}
\newcommand{\ebm}{\end{bmatrix}}

\newcommand{\bmx}{\begin{matrix}}
\newcommand{\emx}{\end{matrix}}

\newcommand{\bpm}{\begin{pmatrix}}
\newcommand{\epm}{\end{pmatrix}}

\newcommand{\btab}{\begin{tabular}}
\newcommand{\etab}{\end{tabular}}

\newsavebox{\myname}
\newsavebox{\emailaddr}

\theoremstyle{plain}
\newtheorem{theorem}{Theorem}[section]

\newtheorem{lemma}[theorem]{Lemma}
\newtheorem{cor}[theorem]{Corollary}

\theoremstyle{definition}
\newtheorem{example}[theorem]{Example}

\date{}
\begin{document}

\title{ Matrix Cubes Parametrized by Eigenvalues}
\author{Jiawang Nie and Bernd Sturmfels\footnote{supported by an Alexander von Humboldt research prize and NSF grant DMS-0456960}}

\maketitle

\abstract{
An elimination problem in semidefinite programming
is solved by means of tensor algebra.
It concerns families of matrix cube problems
whose constraints are the minimum and maximum eigenvalue function
on an affine space of symmetric matrices.
An LMI representation is given for the convex set
of all feasible instances, and its boundary is studied
from the perspective of algebraic geometry.
This generalizes the earlier
work \cite{NPS} with Parrilo on $k$-ellipses and $k$-ellipsoids.
}


\section{Introduction}

The {\em matrix cube problem} in semidefinite programming
is concerned with the following question.
Given real symmetric $N {\times} N$-matrices $A_0, A_1,\ldots, A_m$,
does every point $(t_1,\ldots,t_m)$ in the cube
$\,\prod_{i=1}^m [\lambda_i,\mu_i]\,$
satisfy the matrix inequality
$$ A_0 \,+\, \sum_{k=1}^m t_k A_k \,\, \succeq \,\, 0 \,\, ? $$
The inequality means that the symmetric matrix
$ A_0 + \sum_{k=1}^m t_k A_k $
is positive semidefinite, i.e.~its $N$ eigenvalues are all non-negative reals.
This problem is NP-hard \cite{BN}, and it
 has important applications in robust optimization and control \cite{BNR,NemMC},
e.g., in Lyapunov stability analysis for uncertain dynamical systems,
and for various combinatorial problems which can be reduced to maximizing
a positive definite quadratic form over the unit cube.
For a recent study see \cite{CL}.

In this paper we examine parametrized families of matrix
cube problems, where the lower and upper bounds
that specify the cube are the eigenvalues of symmetric
matrices that range over a linear space
of matrices. We define the set
\be  \label{eq:Cdef}
\mc{C} \,\,\, = \,\,\, \left\{\,
(x,d) \in\re^n \times \re \,\, \left| \,
\baray{l}
d \cdot A_0 + \overset{m}{\underset{k=1}{\sum}}\, t_k A_k \succeq 0
\,\,\,\, \hbox{whenever} \\
\lmd_{\min}(B_k(x)) \leq t_k \leq \lmd_{\max}(B_k(x))
\\ \qquad \qquad \qquad \hbox{for} \,\,\, k=1,2,\ldots, m
\earay
\right.\right\},
\ee
where the $A_i$ are constant symmetric matrices of size $N_0\times N_0$,
the symbols
$\lmd_{\min}(\cdot)$ and $\lmd_{\max}(\cdot)$
denote the minimum and maximum eigenvalues of a matrix, and
the $N_k \times N_k$ matrices $B_k(x)$ are linear matrix polynomials of the form
\[
B_k(x) \,\,\, = \,\,\, B_0^{(k)} + x_1 B_1^{(k)} + \cdots + x_n B_n^{(k)}.
\]
Here $B_0^{(k)}, B_1^{(k)}, \ldots, B_n^{(k)}$ are constant symmetric
$N_k\times N_k$ matrices for all $k$.

Since the minimum eigenvalue function $\lmd_{\min}(\cdot)$ is concave and
the maximum eigenvalue function $\lmd_{\max}(\cdot)$ is convex, we
see that $\mc{C}$ is a convex subset in $\re^{n+1}$.
Our definition of $\,\mc{C}\,$ is by a polynomial system
in free variables $(x,d)$ and universally quantified variables $t$.
Quantifier elimination in real algebraic geometry \cite{BPR} tells us
that $\mc{C}$ is semialgebraic, which means that it can
be described by  a Boolean combination
of polynomial equalities or inequalities in $(x,d)$.
However, to compute such a description by
algebraic elimination algorithms is infeasible.

Linear matrix inequalities (LMI) are a useful and efficient tool
in system and control theory \cite{BEFB}.
LMI representations are convenient
for building convex optimization models,
especially in semidefinite programming \cite{WSV}.
The set described by an LMI is always convex and semialgebraic.
Since our set $\mc{C}$ is convex and semialgebraic,
it is natural to
ask whether $\mc{C}$ admits an LMI representation~?

Our aim is to  answer this question affirmatively.
Theorem~\ref{thm:MainRep} states that
\begin{equation}
\label{LMIrep}
\mc{C}  \,\,\, = \,\,\, \left\{ (x,d) \in\re^n \times \re:\, \mc{L}(x,d) \succeq 0  \right \},
\end{equation}
where $ \mc{L}(x,d) $ is a linear matrix polynomial whose
coefficients are larger symmetric matrices that are constructed
from the matrices $A_i$ and $B^{(k)}_j$.
The construction involves operations from tensor algebra and is carried out in Section~3.

First, however, some motivation is needed.
In Section 2 we shall explain why the convex semialgebraic sets $\mc{C}$
are interesting. This is done by discussing
several geometric applications, notably
the study of $m$-ellipses and $m$-ellipsoids \cite{NPS}.

Section 4 is devoted to algebraic geometry questions
related to our set $\mc{C}$:
What is the Zariski closure $\mc{Z}(\pt\mc{C})$ of the boundary $\pt\mc{C}$~?
What is the polynomial defining $\mc{Z}(\pt\mc{C})$~?
What is the degree of $\mc{Z}(\pt\mc{C})$~?
Is the boundary $\pt\mc{C}$ irreducible~?

In Section 5, we discuss an application to robust control,
explore the notion of matrix ellipsoids, and conclude with some
directions for further research.

\section{Ellipses and beyond}

In this section we illustrate the construction of the set
$\mc{C}$ for some special cases. The first case to
consider is $N_1 = \cdots = N_m = 1$ when each
$B_k(x) = b_k^Tx + \beta_k$ is a linear scalar function.
Then our elimination problem is solved as follows:
\be  \nn
\mc{C} \,\,\, = \,\,\,
\left\{ (x,d) \in \re^{n+1} :\,
d \cdot A_0 +\sum_{k=1}^m (b_k^Tx + \beta_k)  A_k \succeq 0\right\}.
\ee
Thus $\mc{C}$ is a spectrahedron (the solution set of an LMI) and, conversely
every spectrahedron arises in this way.
The Zariski closure $\mc{Z}(\pt\mc{C})$ of its boundary $\pt\mc{C}$
is the hypersurface defined by the vanishing of the determinant of the above matrix.
For generic data $A_i, b_k,\beta_k$ this hypersurface is irreducible
and has degree $N_0$.
Throughout this paper, we shall use the term ``generic'' in the sense
of algebraic geometry (cf.~\cite{NRS}).
Randomly chosen data are generic with probability one.

A natural extension of the previous example is the case when each given matrix
$B_k(x)$ is diagonal. We write this diagonal matrix as follows:
\[
B_k(x) \,\,\, = \,\,\, \diag\big( {b_k^{(1)}}^T \! x + \beta_k^{(1)},
 \ldots ,  \, {b_k^{(N_k)}}^T \! x + \beta_k^{(N_k)} \big)
\]
Then our spectrahedron can be described by an intersection of LMIs:
\be  \nn
\mc{C} \,\,\, = \,\,\,
\left\{ (x,d):\,
d \cdot A_0 +\sum_{k=1}^m ({b_k^{(i_k)}}^T  x + \beta_k^{(i_k)})  A_k \succeq 0, \,\,\,
1\leq i_k \leq N_k , 1\leq k\leq m
\right\}.
\ee
The Zariski closure $\mc{Z}(\pt\mc{C})$ of the boundary $\pt\mc{C}$
is a hypersurface which is typically reducible. It is defined by the product of
all determinants of the above matrices which contribute an
active constraint for $\mc{C}$. Each of the determinants
has degree $N_0$ and there can be as many as
$N_1N_2\cdots N_m$ of these boundary components.

The point of departure for this project was our paper
with Parrilo \cite{NPS} whose result we briefly review.
Given points $(u_1,v_1) , \ldots, (u_m,v_m)$ in the plane $\re^2$
and a parameter $d > 0$, the corresponding {\em $m$-ellipse} $\mc{E}_d$ is the set
of all points $(x_1,x_2)$
whose sum of distances to the given points $(u_i,v_i)$ is at most $d$.
In symbols,
\be  \nn
\mc{E}_d \,\,\, = \,\,\, \left\{ x\in \re^2 \,:\,
 \sum_{k=1}^m  \sqrt{(x_1-u_k)^2 + (x_2-v_k)^2} \,\leq \, d \right\}.
\ee
In \cite{NPS} it is shown that $\mc{E}_d$ is a spectrahedron,
an explicit LMI representation of size $2^m\times 2^m$ is given, and
the degree of $ \partial \mc{E}$ is shown to be
$2^m$ when $m$ is odd, and $2^m - \binom{m}{m/2}$
when $m$ is even.
For instance, if $m=3$, $(u_1,v_1)=(0,0), (u_2,v_2)=(1,0)$ and $ (u_3,v_3)=(0,1)$,
then the $3$-ellipse $\mc{E}_d$ has the LMI representation

\begin{figure}
\centering
 \includegraphics[height=8.3cm]{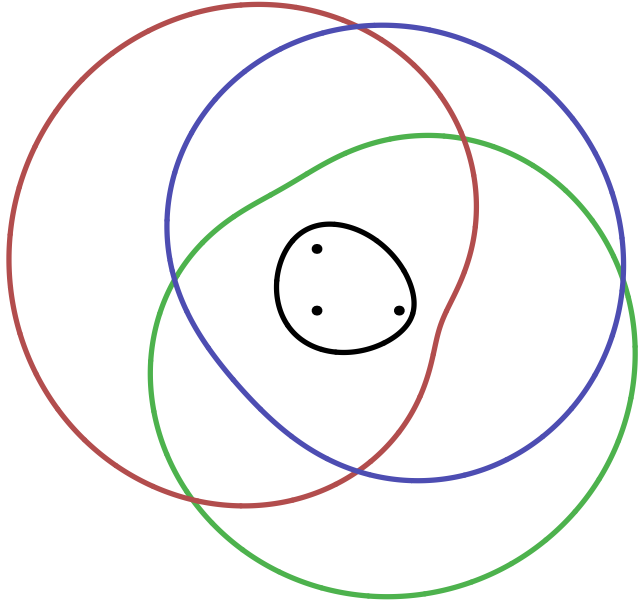}
\vskip -0.4cm
\caption{The Zariski closure of the 3-ellipse is a
curve of degree eight.}
\label{fig:components3ellipse}
\end{figure}

{\scriptsize
 \[\bbm
  d+3x_1-1 \!\! & \!             x_2-1 \! & \!             x_2 \! & \!                   0 \! & \!
     x_2 \! & \!                   0 \! & \!                   0 \! & \!                   0 \\
              x_2-1 \! & \!\!   d+x_1-1 \!\! & \!                   0 \! & \!
                    x_2 \! & \!                   0 \! & \!             x_2\! & \!                   0 \! & \!                   0 \\
              x_2 \! & \!                   0 \! & \!\! d+x_1+1 \! & \!
       x_2-1 \! & \!                   0 \! & \!                   0 \! & \!             x_2 \! & \!                   0 \\
            0 \! & \!             x_2 \! & \!             x_2-1 \! & \!\!  d-x_1+1 \! & \!
                  0 \! & \!                   0 \! & \!                   0 \! & \!             x_2 \\
    x_2 \! & \!       0 \! & \!                   0 \! & \!                   0 \! & \!\!   d+x_1-1 \! & \!
                 x_2-1 \! & \!             x_2 \! & \!                   0 \\
              0 \! & \!             x_2 \! & \!                   0 \! & \!                   0 \! & \!
                x_2-1 \! & \!   d-x_1-1 \!\! & \!                   0 \! & \!             x_2 \\
              0 \! & \!                   0 \! & \!             x_2 \! & \!                   0 \! & \!
 x_2 \! & \!                   0 \! & \!   d-x_1+1 \! \! & \!             x_2-1 \\
          0 \! & \!                   0 \! & \!                   0 \! & \!             x_2 \! & \!
        0 \! & \!             x_2 \! & \!             x_2-1 \! & \!\! \!  d-3x_1+1 \\
 \ebm \succeq 0.
\]}
See Figure \ref{fig:components3ellipse} for a picture of this $3$-ellipse
and the Zariski closure of its boundary.

We can model the $m$-ellipse using matrix cubes as in \reff{eq:Cdef} as follows.
Let $N_0=1$ and define each scalar $A_i$ to be $1$, and
let $N_1 = \cdots = N_m = 2$ and define
\[
B_k(x) \,\, = \,\, \bbm x_1 - u_k & x_2 - v_k \\ x_2 - v_k & u_k - x_1 \ebm \, \quad
\hbox{for} \,\,\, k = 1,\cdots,m.
\]
The eigenvalues of $B_k(x)$ are $\, \pm \sqrt{(x_1-u_k)^2 + (x_2-v_k)^2} $,
and we find that
$$ \mc{C} = \bigl\{(x,d) \in \re^3  \,: \, d \geq t_1  + \cdots + t_m
\,\, \hbox{
whenever $|t_k| \leq \sqrt{(x_1{-}u_k)^2 + (x_2{-}v_k)^2}$} \bigr\} . $$
This formula characterizes the parametric $m$-ellipse
$$ \mc{C} \,\, = \,\, \bigl\{ (x_1,x_2,d) \in \re^3 \,\,: \,\,
(x_1, x_2) \in \mc{E}_d \bigr\}, $$
and its LMI representation (\ref{LMIrep})
is precisely that given in \cite[Theorem 2.3]{NPS}.
The construction extends to $m$-ellipsoids,
where the points lie in a higher-dimensional space.
We shall return to this topic and its algebraic subtleties
in Example \ref{elliexa}.
In Section 5 we introduce a matrix version of
the $m$-ellipses and $m$-ellipsoids.

Consider now the case
when $N_0=1$ and $A_0 = A_1 = \cdots = A_m = -1$
but the $B_k(x)$ are allowed to be
arbitrary symmetric $N_k \times N_k$-matrices
whose entries are linear in $x$. Then
the spectrahedron in \reff{eq:Cdef} has the form
$$
\mc{C} \,\,\, = \,\,\,
\bigl\{
(x,d) \in \re^n \times \re \,:\,
\sum_{k=1}^m \lmd_{\max}(B_k(x)) \leq d
\bigr\}.
$$
A large class of important convex functions on $\re^n$ can be
represented in the form $\,x \mapsto  \lmd_{\max}(B_k(x)) $.
Our main result in the next section gives a recipe for
constructing an explicit LMI representation for the graph of
a sum of  such convex functions. The existence
of such a representation is not obvious, given that
the Minkowski sums of two spectrahedra is
generally not a spectrahedron \cite[\S 3.1]{RG}.

\section{Derivation of the LMI representation} \label{sec:lmi}
\setcounter{equation}{0}

In this section we show that the convex set $\mc{C}$
in (\ref{eq:Cdef}) is a spectrahedron, and we apply tensor operations
as in \cite{NPS} to derive an explicit LMI representation.
We begin with a short review of tensor sum and tensor product
of square matrices.
Given two square matrices $F =( F_{ij})_{1\leq i,j\leq r}$ and
$G =( G_{k\ell})_{1\leq k,\ell\leq s}$,
their standard {\it tensor product $\otimes$} is the block matrix
of format $rs \times rs$ which defined as
\[
F \otimes  G \,\,\, =  \,\,\, \big( F_{ij} G)_{1\leq i,j\leq r}.
\]
Based on tensor product $\otimes$, we define the {\it tensor sum $\oplus$}
to be the $rs \times rs$-matrix
\[
F \oplus G \,\,\,= \,\,\, F \otimes  I_s + I_t \otimes G.
\]
Here $I_r$ denotes the identity matrix of size $r\times r$.
For instance, the tensor sum of two $2 \times 2$-matrices
is given by the formula

\[
\bbm a & b \\ c & d \ebm \oplus
\bbm e & f \\ g & h \ebm \,\,\, = \,\,\,
 \bbm
\bmx a+e & f \\ g & a+h \emx  &
\bmx  b & 0 \\ 0 & b \emx  \\
\bmx  c & 0 \\ 0 & c \emx  &
\bmx d+e & f \\ g & d+h \emx
\ebm.
\]
Tensor products and tensor sums of matrices are also known as
{\it Kronecker products} and {\it Kronecker sums} \cite{HJ}.
Note that the tensor operations $\otimes$ and $\oplus$
are not commutative but they are associative. Hence we
can remove parentheses when we take the
tensor product or tensor sum of $k$ matrices in a fixed order.

The eigenvalues of $F \otimes G$ are the products of
pairs of eigenvalues of $F$ and $G$. Similarly, the
eigenvalues of $F \oplus G$ are the
sums of such pairs. This is well-known for
tensor products, but perhaps slightly less
so for tensor sums. We therefore explicitly state the following
lemma on diagonalization of tensor sums.

\begin{lemma} {\rm \cite[Lemma 2.2]{NPS} } \label{lem:congru}
Let $M_1,\ldots,M_k$ be symmetric matrices,
$U_1,\ldots,U_k$ orthogonal matrices,
and $\Lambda_1,\ldots,\Lambda_k$ diagonal matrices
such that $\,M_i = U_i \cdot \Lambda_i \cdot U_i^T\,$ for $ i=1,\ldots,k$. Then
the tensor sum transforms as follows:
\[
(U_1 \otimes \cdots \otimes U_k)^T \cdot
(M_1 \oplus \cdots \oplus M_k) \cdot
(U_1 \otimes \cdots \otimes U_k) \,\,=\,\,
\Lambda_1 \oplus \cdots \oplus \Lambda_k.
\]
In particular, the eigenvalues of  $M_1 \oplus M_2 \oplus \cdots \oplus M_k$
are the sums $\lambda_1 + \lambda_2 + \cdots + \lambda_k$ where
$\lambda_1$ is any eigenvalue of $M_1$,
$\lambda_2$ is any eigenvalue of $M_2$, etc.
\end{lemma}

We now turn to the construction of an LMI representation for
the convex semialgebraic set $\mc{C}$.
First, we introduce a linear operator $\mc{A}$
on the space of linear matrix polynomials in $m+1$
unknowns  $a=(a_0,a_1,\ldots,a_m)$. Namely,
let $P_a(x,d)$ be such a linear matrix polynomial of the form
\begin{equation}
\label{Pdef}
P_a(x,d) \,\,\, = \,\,\, a_0P_0(x,d)+a_1P_1(x,d)+\cdots+a_mP_m(x,d)
\end{equation}
where the coefficients $P_0(x,d),\ldots,P_m(x,d)$ are matrix polynomials in $(x,d)$.
Using the matrices $A_0,A_1,\ldots,A_n$ in (\ref{eq:Cdef}), we define
the linear operator $\mc{A}$ by
\[
\mc{A}\bigl(P_a(x,d)\bigr) \,\, :=
\,\, P_0(x,d)\otimes A_0+P_1(x,d)\otimes A_1+\cdots+ P_m(x,d)\otimes A_m.
\]
Second, using the matrices $B_k(x)$ in (\ref{eq:Cdef}) we define
the following tensor sum
\be \nn
L_a(x,d) \,\, := \,\,  ( d a_0) \oplus
(a_1 B_1(x))  \oplus  \, \cdots \, \oplus  (a_mB_m(x)).
\ee
This expression is linear matrix polynomial in  $a=(a_0,a_1,\ldots,a_m)$,
and we set
\begin{equation}
\label{DefOfL}
\mc{L}(x,d) \,\,\, := \,\,\, \mc{A} \big(L_a(x,d)\big).
\end{equation}
Since $L_a(x,d)$ is linear in both $(x,d)$ and $a$,
and since $\mc{A}$ is a linear operator, the matrix
$\mc{L}(x,d)$ depends linearly on $(x,d)$ and on
the matrices $A_0,A_1,\ldots,A_m$.

\begin{theorem}  \label{thm:MainRep}
The convex semialgebraic set $\mc{C} \subset \re^{n+1}$
in (\ref{eq:Cdef}) is a spectrahedron,
and it can be represented by the LMI~\reff{LMIrep}.
\end{theorem}

This theorem implies that
 optimizing a linear function over the convex set $\mc{C}$ is
an instance of semidefinite programming \cite{WSV}.
For instance, for a fixed value $d^*$ of $d$,
to minimize a linear functional $c^Tx$ over its section $\,\mc{C} \cap \{\,(x,d) : d = d^* \}\,$
is equivalent to solving the semidefinite programming problem
\[
\mbox{minimize} \,\,\, c^Tx \quad \mbox{subject to} \quad  \mc{L}(x,d^*) \succeq 0.
\]
This problem can be solved by standard SDP solvers like {\tt SeDuMi} \cite{sedumi}.

To prove Theorem~\ref{thm:MainRep} we need one more fact
concerning the linear operator $\mc{A}$, namely, that
$\mc{A}$ is invariant under congruence transformations.

\begin{lemma}  \label{lem:Acngr}
Let $P_a(x,d) $ be as in (\ref{Pdef}).
For any matrix $U(x,d)$ we have
\begin{equation}
\label{lemdis} \!
\mc{A}\bigl( U(x,d)^TP_a(x,d)U(x,d)\bigr) =
(U(x,d)\otimes I_{N_0})^T \mc{A}\bigl(P_a(x,d)\bigr) (U(x,d)\otimes I_{N_0}).
\end{equation}
\end{lemma}

\begin{proof}
First note that tensor product satisfies $(A\otimes B)^T= A^T \otimes B^T$ and
\[
(M_1\otimes M_2)\cdot(M_3\otimes M_4)\cdot(M_5\otimes M_6)
= (M_1\cdot M_3\cdot M_5)\otimes (M_2\cdot M_4\cdot M_6)
\]
Using these identifies we perform the following direct calculation:
\begin{align*}
 &  \Big(U(x,d)\otimes I_{N_0}\Big)^T\, \mc{A}\Big(P_a(x,d)\Big)\, \Big(U(x,d)\otimes I_{N_0}\Big)  \\
= \,\,\, & \Big(U(x,d)\otimes I_{N_0}\Big)^T \left(\sum_{k=0}^m P_k(x,d)\otimes A_k \right) \Big(U(x,d)\otimes I_{N_0}\Big) \\
= \,\, & \sum_{k=0}^m  \Big(U(x,d)\otimes I_{N_0}\Big)^T \Big( P_k(x,d)\otimes A_k\Big) \Big(U(x,d)\otimes I_{N_0}\Big)  \\
= \,\, & \sum_{k=0}^m  \Big(U(x,d)^TP_k(x,d)U(x,d)\Big)\otimes A_k.
\end{align*}
By definition of $\mc{A}$, this expression is equal to the left hand side of (\ref{lemdis}).
\end{proof}

\begin{proof}[Proof of Theorem~\ref{thm:MainRep}]
Let $U_k(x)$ be orthogonal matrices such that
\begin{equation}
\label{UandLambda}
U_k(x)^T \cdot B_k(x) \cdot U_k(x) \,\,\, =\,\,\,
 \diag \left( \lmd^{(k)}_1(x), \ldots, \lmd^{(k)}_{N_k}(x) \right)
\, =: \, D_k .
\end{equation}
Here $\lmd^{(k)}_j(x)$ are the algebraic functions representing eigenvalues
of the matrices $B_k(k)$. If we set
$\, Q \,=\,  (1) \otimes \, U_1(x)  \,  \otimes  \,\cdots \, \otimes \, U_m(x)\,$ then,
by Lemma~\ref{lem:congru}, we have
\[
Q^T \cdot L_a(x,d) \cdot Q \,\,\, = \,\,\, (d a_0) \oplus
(a_1 D_1(x)  ) \, \oplus  \,\cdots\, \oplus (a_m D_m(x)) \,\, =:\,\, \tilde L.
\]
Note that $\tilde L$ is a diagonal matrix, with each diagonal entry having the form
\[
d\cdot a_0 \,+\, \sum_{k=1}^m  \lmd^{(k)}_{j_k}(x) a_k .
\]
It follows that $\mc{A} (\tilde L)$
is a block diagonal matrix, with each block of the form
\begin{equation}
\label{Amatrix}
d\cdot A_0 \,+\, \sum_{k=1}^m  \lmd^{(k)}_{j_k}(x) A_k .
\end{equation}
By Lemma~\ref{lem:Acngr} and the definition of $\mc{L}(x,d) $, we have
\begin{equation}
\label{Qtransformation}
(Q\otimes I_{N_0})^T \cdot \mc{L}(x,d) \cdot (Q\otimes I_{N_0}) \,\,\, = \,\,\,
\mc{A} \Big(Q^T \cdot L_a(x,d) \cdot Q\Big) \,\,\, = \,\,\, \mc{A} \big(\tilde L\big).
\end{equation}
Hence $\mc{L}(x,d)\succeq 0$ if and only if the $N_1 N_2 \cdots N_m$ blocks
(\ref{Amatrix}) are simultaneously positive semidefinite for all
 index sequences
$j_1,j_2,\ldots,j_m$.
Given that the $\lmd^{(k)}_{j_k}(x)$ are the eigenvalues of $B_k(x)$ we
conclude that
 $\,\mc{L}(x,d)\succeq 0\,$ if and only if
\[
d\cdot A_0 + \sum_{k=1}^m  t_k   A_k \succeq 0 \quad
\hbox{for all} \,\, (t_1,\ldots,t_k) \,\,
\hbox{such that
$t_k$ is eigenvalue of $B_k(x)$.}
\]
Since $\,d\cdot A_0 + \sum_{k=1}^m  t_k   A_k \succeq 0\,$
describes a convex set in $t$-space,
we can now replace the condition
``$t_k$ is an eigenvalue of $B_k(x)$''
by the equivalent condition
\[
\lmd_{\min}(B_k(x)) \leq t_k \leq \lmd_{\max}(B_k(x))
 \quad \hbox{for} \,\,\, k=1,2,\ldots, m.
\]
We conclude that   $\,\mc{L}(x,d)\succeq 0\,$ if and only if
$(x,d) \in \mc{C}$.
\end{proof}

The transformation from the given symmetric matrices $A_i$ and $B_j(x)$
to the bigger matrix $\mc{L}(x,d) $ is easy to implement in {\tt Matlab}, {\tt Maple}
or {\tt Mathematica}. Here is a small explicit numerical example
which illustrates this transformation.

\begin{example} \label{Ex22222}
Let $n=m=N_0 = N_1 = N_2 = 2$ and consider the input matrices
\[
A_0 \,= \,\bbm    2  &   1 \\    1   &  2 \ebm, \quad
A_1 \,= \,\bbm    1  &   1 \\    1   &  0 \ebm, \quad
A_2 \,= \,\bbm    0  &   1 \\    1   &  1 \ebm,\qquad \hbox{and}
\]
\[
B_1(x) = \bbm
  3-x_1+2 x_2 & 2 x_1-x_2-2 \\
 2 x_1-x_2-2 &    -1+2 x_1
\ebm \! , \,\,
B_2(x) = \bbm
     2+x_1 & 1+3 x_1-x_2 \\
 1+3 x_1-x_2 & 3-2 x_1+x_2
\ebm.
\]
Then the three-dimensional spectrahedron $\mc{C}$ is represented by the LMI
{\scriptsize
\begin{align*}
\left[\baray{rrrr}
 2d+3-x_1+2x_2& d+5+2x_2  &   0&       1+3x_1-x_2 \\
       d+5+2x_2&       2d+2+x_1&     1+3x_1-x_2&     1+3x_1-x_2 \\
  0&       1+3x_1-x_2& 2d+3-x_1+2x_2& d+6-3x_1+3x_2   \\
     1+3x_1-x_2&     1+3x_1-x_2& d+6-3x_1+3x_2& 2d+3-2x_1+x_2  \\
    -2+2x_1-x_2&    -2+2x_1-x_2  & 0& 0 \\
    -2+2x_1-x_2  &  0&   0& 0 \\
   0&    0&      -2+2x_1-x_2&    -2+2x_1-x_2   \\
    0&   0&      -2+2x_1-x_2  &        0
\earay   \right. \qquad \qquad \qquad & \\
\qquad \qquad \qquad \left.\baray{rrrr}
-2+2x_1-x_2&    -2+2x_1-x_2  &   0&   0\\
-2+2x_1-x_2  &  0&    0& 0\\
0&   0&      -2+2x_1-x_2&    -2+2x_1-x_2\\
 0&   0&    -2+2x_1-x_2  &  0\\
  2d-1+2x_1&       d+1+3x_1  &  0&       1+3x_1-x_2\\
      d+1+3x_1&       2d+2+x_1&     1+3x_1-x_2&     1+3x_1-x_2\\
 0&       1+3x_1-x_2&     2d-1+2x_1& d+2+x_2\\
  1+3x_1-x_2&   1+3x_1-x_2&   d+2+x_2& 2d+3-2x_1+x_2
\earay  \right] & \, \succeq \, 0.
\end{align*}
}
Note that the boundary of $\mc{C}$ is a surface of degree
eight in $\re^3$. If we fix a positive real
value for $d$, then this LMI describes a two-dimensional spectrahedron.\qed
\end{example}

\section{Algebraic degree of the boundary}
\label{sec:geom}
\setcounter{equation}{0}

We have seen in Theorem  \ref{thm:MainRep}
that the set  $\mc{C}$ is a spectrahedron. This
implies the property of {\it rigid convexity}, which was
introduced by Helton and Vinnikov \cite{HV}.
We  briefly review this concept, starting with a
discussion of {\it real zero} polynomials.
Let $f(z)$ be a polynomial in $s$ variables
$(z_1,\ldots,z_s)$ and $w$ a point in $\re^s$.
We say $f(z)$ is a {\it real zero (RZ) polynomial with respect to $w$}
if the univariate polynomial \[ g(\af) \,\, := \,\, f(w + \af \cdot v ) \]
has only real roots for any nonzero vector $v\in \re^n$.
Being real zero with respect to $w$ is a necessary condition \cite{HV}
for $f(z)$ to have a determinantal representation
\[ f(z) \,\,\, = \,\,\, \det (F_0 + z_1 F_1 + \cdots + z_s F_s ),\]
where the $F_i$ are constant symmetric matrices such that
$\,F_0 + w_1 F_1 + \cdots + w_s F_s  \succ 0 $.
Helton and Vinnikov \cite{HV} showed that
the converse is true when $s=2$, but a similar
converse is not known  for $s>2$.
For representations of convex sets
as projections of spectrahedra in higher dimensional space
we refer to \cite{HN1,HN2}.

A convex set $\Gamma$ is called {\em rigid convex}
if $\Gamma$ is a connected component of the set
$\{z \in\re^s:\, f(z)>0\}$
for some polynomial $f(z)$ that is real zero
with respect to some interior point of $\Gamma$.
As an example, the unit ball $\{z\in\re^n:\, \|z\|\leq 1\}$
is a rigid convex set. Note that, not every convex semialgebraic set is rigid convex.
For instance, the set $\{z\in\re^2:\, z_1^4+z_2^4\leq 1\}$
is not rigid convex as shown in \cite{HV}.

The boundary of the $3$-ellipse in
Figure \ref{fig:components3ellipse} is rigid convex
of degree eight,
and the curve (for fixed $d$) in Example \ref{Ex22222}
is also a rigid convex curve of degree eight.
In light of Theorem \ref{thm:MainRep}, our discussion implies
the following corollary.

\begin{cor}
The convex semialgebraic set $\mc{C} \subset \re^{n+1}$ in (\ref{eq:Cdef}) is
rigid convex.
\end{cor}

We now examine the algebraic properties of the boundary of the
rigid convex set $\mc{C}$, starting from the explicit
LMI representation $\mc{L}(x,d) \succeq 0 $
which was constructed in (\ref{DefOfL}).
Since the tensor product is a bilinear operation,
the entries of the symmetric matrix $\mc{L}(x,d)$
are linear in the $n+1$ variables
$(x,d)$, and they also depend linearly on the entries
of the constant matrices $A_i$ and $B_{j}^{(k)}$.
We regard these entries as unknown parameters,
and we let $K = \Q(A,B)$ denote the field extension
they generate over the rational numbers $\Q$.
Then $K[x,d]$ denotes the ring of polynomials in the
$n+1$ variables $(x,d)$ with coefficients in the
rational function field $K=\Q(A,B)$.
 Our object of interest is the determinant
\[ r(x,d) \,\,\, = \,\,\, \det \, \mc{L}(x,d).  \]
This is an element of $K[x,d]$. By
a specialization of $r(x,d)$ we mean the
image of $r(x,d)$ under any field homomorphism
$K \rightarrow \re$. Thus a {\em specialization} of $r(x,d)$ is
a polynomial in $\re[x,y]$ which arises as the determinant
of the LMI representation (\ref{LMIrep})
for some specific matrices $A_i$ and $B_j^{(k)}$
with entries in $\re$.

\begin{theorem}
The polynomial $r(x,d)$ has degree $N_0 N_1 \cdots N_m$,
and it is irreducible as an element of $K[x,y]$.
Generic specializations of $r(x,d)$ are irreducible polynomials
of the same degree in $\re[x,d]$. Therefore,  if the given symmetric
matrices $A_i$ and $B_j^{(k)}$ are generic,
then the boundary of the spectrahedron $\mc{C}$
is a connected component of an irreducible real
hypersurface of degree $N_0 N_1 \cdots N_m$.
\end{theorem}

\begin{proof}
As in (\ref{UandLambda}) we write
$\lambda_1^{(k)}(x),\ldots,\lambda_{N_k}^{(k)}(x)$ for the eigenvalues of the
matrix $B_k(x)$. These eigenvalues are elements in the
algebraic closure of the rational function field $\Q(B,x)$.
Here the entries of the matrix coefficient of $x_i$ in $B_k(x)$
are regarded as variables over $\Q$. This implies
that the characteristic polynomial of $B_k(x)$ is
irreducible over $\Q(B,x)$ and the Galois group \cite{lang}
of the corresponding algebraic field extension
$\Q(B,x) \subset \Q(B,x)(\lambda_1^{(k)}(x),\ldots,\lambda_{N_k}^{(k)}(x))$
is the full symmetric group $S_{N_k}$ on $N_k$ letters.
Let $L$ be the algebraic field extension of $\Q(B,x)$
generated by all eigenvalues $\lambda_i^{(k)}(x)$
for $k = 1,\ldots,m$ and $i = 1,\ldots,N_k$. Since
all scalar matrix entries are independent indeterminates over $\Q$,
we  conclude that $\Q(B,x) \subset L$ is an algebraic field extension
of degree $N_1 \cdots N_m$, and the Galois group of this
field extension is the product of symmetric groups
$\,S_{N_1} \! \times \cdots \times S_{N_m}$.

The square matrix $\mathcal{L}(x,d)$ has
$N_0 N_1 \cdots N_m$ rows and columns, and each
entry is a linear polynomial in $(x,d)$ with coefficients
in $\Q(A,B)$. The polynomial $r(x,d)$ is the determinant of this
matrix, so it has degree $N_0 N_1 \cdots N_m$ in $(x,d)$.
To see that it is irreducible, we argue as follows.
In light of (\ref{Qtransformation}), the matrices $\mc{L}(x,d)$
and $\mc{A} \big(\tilde L\big)$ have the same determinant,
and this is the product of the determinants of the
$N_1 N_2 \cdots N_m$ blocks (\ref{Amatrix}) of size $N_0 \times N_0$.
We conclude
$$ r(x,d) \,\,\, = \,\,\,  \prod_{j_1=1}^{N_1} \prod_{j_2=1}^{N_2} \!\! \cdots \!\!
\prod_{j_m=1}^{N_m} \!
{\rm det} \bigl( d\cdot A_0 \,+\,
\lmd^{(1)}_{j_1}(x) A_1 +
\lmd^{(2)}_{j_2}(x) A_2 + \cdots +
\lmd^{(m)}_{j_m}(x) A_m   \bigr). $$
The Galois group mentioned above acts transitively by permuting the factors in this
product. No subset of the factors is left invariant under this permutation
action. This proves that $r(x,d)$ is irreducible as univariate polynomial
in  $\Q(A,B,x)[d]$, and hence also as an $(m+1)$-variate polynomial in
$\Q(A,B)[x,d]$.

The assertion in the third sentence follows because the property
of a polynomial with parametric coefficients to be irreducible is
Zariski open in the parameters. If we specialize the
entries of $A_i$ and $B_j^{(k)}$ to be random
real numbers, then the resulting specialization of
$r(x,d)$ is an irreducible polynomial in $\re[x,d]$.
\end{proof}

Naturally, in many applications the entries of $A_i$ and $B_j^{(k)}$
will not be generic but they have a special structure. In those cases,
the characteristic polynomial of $B_k(x)$ may not be irreducible,
and a more careful analysis is required in determining the degree
of the hypersurface bounding the spectrahedron $\mc{C}$.
Let $h_k(\alpha) $ be the univariate polynomial
of minimal degree in $ \Q(B,x)[\alpha]$ such
that
\[
h_k(\lmd_{\min}(B_k(x))) \,\,= \,\, h_k(\lmd_{\max}(B_k(x))) \,\,= \,\, 0.
\]
holds for all points $x$ in some open subset of $\re^n$.
Thus, $h_k(\alpha)$ is the factor of the characteristic
polynomial $\,{\rm det}(B_k(x) - \alpha I_{N_k})\,$
which is relevant for our problem.
By modifying the argument in the previous proof,
we obtain the following result.

\begin{theorem}
\label{htheorem}
The Zariski closure of the boundary of the spectrahedron $\mc{C}$
is a possibly reducible real hypersurface in $\re^{n+1}$.
If the matrix $A_0$ is invertible then
the degree of the polynomial in $(x,d)$ which defines this hypersurface equals
\begin{equation}
\label{hnumber}
N_0 \cdot \prod_{k=1}^m {\rm degree}(h_k)
\end{equation}
\end{theorem}

It is important to note that, for fixed $d$, the degree of the
resulting polynomial in $x$ can be smaller than
(\ref{hnumber}). Thus this number is only an upper
bound for the algebraic degree of the boundary of the
spectrahedron $\{x \in \re^n : \mc{L}(x,d) \geq 0 \}$.

\begin{example}
\label{elliexa}
The classical ellipse consists of
all points $(x_1,x_2) \in \re^2$ whose sum of distances
to two given points $(u_i,v_i)$ is $d$.
In the formulation sketched in Section 2
and worked out in detail in \cite{NPS}, we obtain
$N_0 = 1$ and $N_i = {\rm degree}(h_k) = 2$ for $k=1,2$.
Hence the number (\ref{hnumber}) is four, and this is indeed
the degree of the surface $\partial \mc{C} \subset \re^3$.
Yet, the boundary of the ellipse is a curve of degree two. \qed
\end{example}

\section{Applications and questions} \label{sec:disc}
\setcounter{equation}{0}

The matrix cube defined in \reff{eq:Cdef}
has important applications in robust control
\cite{BN, BNR, NemMC}.
Consider nonlinear feedback synthesis for the discrete dynamical system
\be \label{sys:NLTV}
\left.
\baray{ll}
x(k+1) \,\,= \,\, F(u(k)) \cdot x(k), & \, x(k) \in \re^n \\
u(k+1) \,\, = \,\, f(x(k))) , & \, u(k) \in \re^m
\earay \right\}
\ee
where $x(k)$ is the state and $u(k)$ is the feedback.
We assume that the matrix $\,F(u)=F_0 + u_1F_1 + \cdots + u_mF_m \in \re^{r\times s}\,$ is linear in $u$
and that $f(x)$ is an algebraic function of $x = (x_1,\ldots,x_m)$.
Such a function $f(x)$ can be implicitly
 defined by eigenvalues of linear matrices $B_k(x)$, that is,
 the coordinates of  $u=f(x)$ satisfy
\[
\det \big( u_k \cdot I_{N_k} - B_k(x) \big) = 0, \,\,\,\, k=1,\ldots, m.
\]
In stability analysis one wish to know
what states make the system \reff{sys:NLTV} stable.
The task is to identify states $x$ such that the
feedback $u$ satisfies $\|F(u)\|_2 \leq 1$.
Equivalently, we need to find $x$ such that
the eigenvalues $u_k$ of all $B_k(x)$ satisfy
\[
\bbm I_r & F(u) \\ F(u)^T & I_s \ebm \,\preceq \, I_{r+s}.
\]
We conclude that such states $x$
are the elements of a convex set of the form \reff{eq:Cdef}
when every $B_k(x)$ can be chosen to be symmetric.
Thus Theorem~\ref{thm:MainRep} furnishes
an LMI representation for the set of stable states.

\smallskip

In Section 2 we have seen that the LMI representation
of multifocal ellipses in \cite{NPS} arise as a special case
of our construction. We next propose a natural matrix-theoretic
generalization of this, namely, we shall define the
{\it matrix ellipsoid} and the {\it matrix $m$-ellipsoid}.
Recall that an $m$-ellipsoid is the subset of $\re^n$ defined as
\[
\Big\{x \in \re^n :\, a_1 \| x - u_1 \|_2 + \cdots + a_m \|x-u_m\|_2 \leq d \Big\}
\]
for some constants $a_1,a_2,\ldots,a_m, d >0$
and fixed foci $u_1,u_2,\cdots,u_m \in \re^n$.
We define the {\em matrix $m$-ellipsoid} to be the convex set
\begin{equation}
\label{matrixellipsoid}
\Big\{x \in \re^n :\, A_1 \| x - u_1 \|_2  + \cdots + A_m \|x-u_m\|_2
\, \preceq \, d\cdot I_N \Big\},
\end{equation}
for some $d>0$, fixed foci $u_1,u_2,\cdots,u_m \in \re^n$,
and positive definite symmetric $N$-by-$N$ matrices
$A_1,\ldots,A_m \succ 0$.
To express the matrix $m$-ellipsoid (\ref{matrixellipsoid})
as an instance of the spectrahedron $\,\mathcal{C}\,$ in  (\ref{eq:Cdef}),
we fix $A_0 = I_N$ and the linear matrices
\[
B_k(x) \quad = \quad
\bbm 0 & x_1 {-} u_{k,1} &  \cdots & x_n {-} u_{k,n} \\
\, x_1 - u_{k,1} &       0              &  \cdots &         0         \\
\,\vdots    & \vdots           &  \ddots  &  \vdots      \\
\,x_n - u_{k,n} & 0                 &  \cdots &         0         \\
\ebm.
\]
Theorem \ref{thm:MainRep} provides an LMI representation of
the matrix $m$-ellipsoid (\ref{matrixellipsoid}).
This generalizes the LMI given in
\cite[\S 4.2]{NPS} for the case when all $A_i$ are scalars.
By Theorem~\ref{htheorem}, the boundary of (\ref{matrixellipsoid})
is a hypersurface of degree at most $N 2^m$.

\smallskip

One obvious generalization of our
parametric matrix cube problem is the set
\be  \nn
\tilde{C} :=\left\{ (x,d) \left| \,
\baray{r}
d \cdot A_0 + \overset{m}{\underset{k=1}{\sum}}\, t_k A_k \succeq 0,\,
\forall \,(t_1,\cdots,t_m):  \\
\lmd_{\min}(B_k(x)) \leq t_k \leq \lmd_{\max}(E_k(x)) \\
\mbox{for } k=1,\cdots, m
\earay
\right. \right\}
\ee
where $E_k(x)$ are linear symmetric matrices different from $B_k(x)$
for $k=1,\ldots,m$.
Since the minimum eigenvalue function $\lmd_{\min}(\,\cdot \,)$ is concave
and the maximum eigenvalue function $\lmd_{\max}(\,\cdot\,)$ is convex,
we see that the set $\tilde{C}$ defined as above is also a convex set.
Assuming the extra hypotheses $\lmd_{\max}(E_k(x))\geq \lmd_{\max}(B_k(x))$
and $\lmd_{\min}(E_k(x))\geq \lmd_{\min}(B_k(x))$,
the convex set $\mc{C}$ can be equivalently defined as
\be  \nn
\left\{ (x,d) \left|\,
\baray{r}
d \cdot A_0 + \overset{m}{\underset{k=1}{\sum}}\, t_k A_k \succeq 0,\,
\forall \,(t_1,\cdots,t_m):  \\
\lmd_{\min}(D_k(x)) \leq t_k \leq \lmd_{\max}(D_k(x)) \\
\mbox{for } k=1,\cdots, m
\earay
\right.\right\},
\ee
where $D_k(x) = \diag \big(B_k(x),E_k(x)\big)$ is a block diagonal matrix.
Therefore, $\tilde{C}$ is a special case of \reff{eq:Cdef}
and  Theorem~\ref{thm:MainRep} furnishes an LMI representation.
We do not know whether the extra hypotheses
are necessary for  $\tilde{C}$ to be a spectrahedron.

\smallskip

An interesting research problem concerning
the matrix cube \reff{eq:Cdef}
is to find the smallest size of an LMI representation.
This question was also raised in \cite[\S 5]{NPS}
for the case of $m$-ellipses.
Theorem~\ref{thm:MainRep} gives an LMI representation
of size $N_0N_1\cdots N_m$.
The degree of the boundary $\pt \mc{C}$
is given by the formula
$N_0 \prod_{k=1}^m \mbox{degree}(h_k)$,
by Theorem~\ref{htheorem}.
If $\mbox{degree}(h_k)$ is smaller than $N_k$,
then the size of the LMI in Theorem~\ref{thm:MainRep}
exceeds the degree of $\pt \mc{C}$.
In this situation, is it possible
to find an LMI for $\mc{C}$ with size smaller than $N_0N_1\cdots N_m$?
When $n=2$ and $d$ is fixed, the projection of $\mc{C}$
into $x$-space is a two dimensional convex set $\mc{C}_x$
described by one LMI of size $N_0N_1\cdots N_m$.
The work of Helton and Vinnikov \cite{HV} shows that
there exists an LMI for $\mc{C}_x$
having size equal to the degree of $\mc{C}_x$.
How can the LMI of Theorem~\ref{thm:MainRep}
be transformed into such a  minimal LMI?

\bigskip \bigskip

\noindent {\bf Authors' addresses:}

\medskip

\noindent Jiawang Nie,
Department of Mathematics,
University of California at San Diego,
La Jolla, CA 92093,
{\tt njw@math.ucsd.edu}

\medskip

\noindent Bernd Sturmfels, Department of Mathematics,
University of California, Berkeley, CA 94720,USA,
{\tt bernd@math.berkeley.edu}

\end{document}